\documentclass[11pt]{amsart}

\usepackage{amssymb}

\usepackage{graphicx}

\usepackage{amsmath}

\usepackage{verbatim}

\makeindex

\input xy
\xyoption{all}

\newtheorem{theorem}{Theorem}[section]

\newtheorem{corollary}[theorem]{Corollary}

\newtheorem{lemma}[theorem]{Lemma}

\newtheorem{proposition}[theorem]{Proposition}
\newtheorem{remark}[theorem]{Remark}

\newcommand{\prskip}{\vspace{8pt}} 
\newcommand{\pfskip}{\vspace{6pt}} 
\newcommand{\sectskip}{\vspace{50pt}} 
\newcommand{\introskip}{\vspace{25pt}} 

\begin{document}

\title{The Nilpotent Regular Element Problem}

\author{P. Ara}
\address{Department of Mathematics,
  Universitat Aut\`onoma de Barcelona, 08193 Bellaterra (Bar\-ce\-lona), Spain}
\email{para@mat.uab.cat}

\author{K.\ C.\ O'Meara}
\address{2901 Gough Street, Apartment 302, San Francisco, CA 94123, USA}
 \email{staf198@uclive.ac.nz}

 \date{11 November, 2015}

\thanks{The first-named author was partially supported by the grants DGI MICIIN
MTM2011-28992-C02-01 and MINECO MTM2014-53644-P}



\maketitle
\begin{abstract}
We use George Bergman's recent normal form for universally adjoining an inner inverse to show that, for general rings, a nilpotent regular element $x$ need not be unit-regular.
This contrasts sharply with the situation for nilpotent regular elements in exchange rings (a large class of rings), and for general rings when all powers of the nilpotent  element $x$ are regular.
\end{abstract}
\vspace{5mm}

Questions concerning nilpotent elements are often central in both linear algebra and ring theory. The problem we shall consider here, of whether a nilpotent (von Neumann) regular element $x$ of a general ring $S$ must be unit-regular, may not have quite reached ``central'' status to date, although its answer was important in the first author's proof of Theorem 4 in \cite{A1} that strongly $\pi$-regular rings (in particular, algebraic algebras over fields) have stable range one. The  problem is also relevant to certain possible direct limit constructions of non-separative regular rings. (This was shown in a privately circulated note by the second author in June 2015. See \cite{A2} for a description of the fundamental Separativity Problem). Our nilpotent regular element problem  is also discussed in the forthcoming book of Lam \cite{L}. Thus it is more than just a \emph{pesky little problem} that has bothered some of us for a number of years. To settle the question, we turn to the recent description by George Bergman \cite{B1} of universally adjoining an inner inverse (quasi-inverse) of an element in an arbitrary algebra over a field. This is possibly the first application of Bergman's lovely result (but surely not the last).
\prskip

In 1996, the first author showed in Theorem 2 of \cite{A1} that nilpotent regular elements of exchange rings must be unit-regular. In 2004, Beidar, Raphael, and the second author showed
in Theorem 3.6 of \cite{B2} that in arbitrary rings, if a nilpotent element has \textbf{\emph{all}} its powers regular then it is unit-regular.
(See Chapter 4 of \cite{ATLA} for a more leisurely account of this result and how its parent result fits into linear algebra.) It is interesting to note that there are even finite-dimensional algebras in which nilpotent regular elements don't have all their powers regular (see Yu \cite{Y}). But in our case, the first possible case of a nilpotent regular element $x$ of an algebra $S$ that is not unit-regular requires $x$ to be of index at least  3 (otherwise its powers are regular) and $S$ to be  infinite-dimensional and lacking the manners of ``good'' algebras. Of course, with many problems in ring theory, if there are counter-examples, then there must be a ``free'' one, $S$.  However, without a good normal form for the members of $S$, viewing the free object as the solution can be delusional! Fortunately for us, the free object we use has such a nice normal form.
\prskip

To construct our counter-example, we apply Bergman's  normal form \cite{B1} in the following situation. Start with the algebra $R = F[x]/(x^3)$ where $x$ is an indeterminate and  $F$ is any field. We identify $x$ with its image in $R$ (whence $x$ is a nilpotent element of index 3). Next let $S = R\langle q \, | \, xqx = x, qxq = q \rangle$ be the algebra obtained from $R$ by freely adjoining a generalised inverse $q$ of $x$. We use the normal form to show that $x$ is not unit-regular in $S$.
\prskip

After we submitted an earlier version of our paper, we learned that Pace Nielsen and Janez \v{S}ter have independently (and at about the same time) also discovered an example of a nilpotent regular element that is not unit-regular. At first glance, their method appears quite different to ours: start with the algebra $R = F\langle a, b \, | \, a^2 = 0 \rangle$ and its left ideal $I = R(1 - ba)$, and  form the subalgebra $T$ of $M_2(R)$ given as
\[
    T \ = \ \left[\begin{array}{cc}
               R & I \\
               R & F+I
             \end{array}\right].
\]
By a clever argument, the authors show directly that for the elements
\[
 X \, = \, \left[\begin{array}{cc}
               a & 0 \\
               1 & 0
             \end{array}\right], \ \
   Q \, = \, \left[\begin{array}{cc}
               b & 1-ba \\
               0 & 0
             \end{array}\right]
\]
$X$ is nilpotent of index 3 and $Q$ is a generalized inverse of $X$, but $X$ is not unit-regular in $T$. See Example 3.19 in \cite{N}. Surprisingly,  the two algebras $S$ and $T$ are actually isomorphic under the correspondence $x \mapsto X,\ q \mapsto Q$. This we show in Section 3.
\prskip

The Nielsen--\v{S}ter argument is shorter than our original one, and in Section 3 we give an even shorter argument still but in a similar spirit to theirs. On the other hand, we feel our method was a \emph{surer bet}. Indeed, we quickly realised on first seeing Bergman's preprint in February 2015 that, with a fair measure of confidence, we could use Bergman's normal form to resolve the regular nilpotent element problem, \emph{one way or the other}. (Without the observation that the algebras $S$ and $T$ are isomorphic, this luxury is missing using $T$ in isolation.) Also, we hope our method is a flag-bearer for George Bergman's results on adjoining a universal inner inverse, which have the potential to be used for attacking other problems such as the fundamental Separativity Problem for regular rings.
\sectskip

\section{Preliminaries}
\introskip

We refer the reader to Goodearl's  book \cite{G1} and the upcoming book by Lam \cite{L} for background on (von Neumann) regular rings and related element-wise properties in more general general rings.  Thus, in a general ring $R$ \,(assumed associative with identity), an element $a \in R$ is \textbf{(von Neumann) regular} if there is an element $b \in R$ such that $a = aba$. Following Lam, we call any such $b$ an \textbf{inner inverse} of $a$. The more established term is ``quasi-inverse'' (and there are also competing terms for this within linear algebra), but Lam's term is perhaps more  suggestive and does not conflict with other uses of quasi-inverse in ring theory. If $a = aba$ and $b = bab$ (so that $a$ is also an inner inverse of $b$), we shall call any such $b$ a \textbf{generalised inverse} of $a$. Again, there are competing terms for this. Notice that if $b$ is an inner inverse of $a$, then $bab$ is a generalised inverse of $a$. If there exists an inner inverse $u$ of $a \in R$ that is a unit, then we say $a$ is \textbf{unit-regular} in $R$.
\prskip

Suppose $a$ is a regular element of a ring $R$ and $b$ is an inner inverse of $a$. Let $e = ab$ and $f = ba$. Then $e$ and $f$ are idempotents with $aR = eR$ and $Rf = Ra$, whence $(1 - e)R$ is a complement of $aR$ and $(1 - f)R$ is the right annihilator ideal of $a$. Unit-regularity of $a$ is equivalent to $(1 - f)R \cong (1 - e)R$ \,(``kernel isomorphic to cokernel''). See the proof of Theorem 4.1 in \cite{G1}. In turn, the latter condition is equivalent to the existence of $c \in (1 - e)R(1 - f)$ and $d \in (1 - f)R(1 - e)$ such that $cd = 1 - e$ and $dc = 1 - f$.
\prskip

Now let us introduce our ring $S$ and Bergman's normal form for its members. We start with the polynomial ring $F[x]$ in the indeterminate $x$ and over a field $F$, and take $R = F[x]/(x^3)$. Thus we may regard $R$ as the 3-dimensional algebra over $F$ containing a nilpotent element $x$ of index 3 and with a basis $\{1, x, x^2\}$. Let $S = R\langle q \, | \, xqx = x, qxq = q \rangle$ be the algebra obtained from $R$ by freely adjoining a generalised inverse $q$ of $x$.
\prskip

\begin{proposition}\label{P:normal form}
The algebra $S$ has a basis $\mathcal{A}$ (over $F$) consisting of $1$ and words (products) alternating in powers $x^i$, for $i = 1,2$, and $q^j$ for $j \ge 1$ \,(either power can begin or end) but with the restriction that a  power of $x$ or $q$ to exponent one can occur only at the beginning or end of a word. For instance, $q^3x^2q^2x$ is a basis word as described whereas $x^2q^4xq^2$ is not (without further reduction).
\end{proposition}
\pfskip

\begin{proof} This is a direct application of Bergman's Corollary 19 of \cite{B1} (see also Lemma 20)  to the following situation. Start with the algebra $R = F[x]/(x^3)$ and the basis $B \cup \{1\}$ for $R$, where $B = B_{++} = \{x, x^2\}$. Fix  the element $p = x$ and note that $1 \notin pR + Rp$. Then the basis $\mathcal{A}$ for our algebra $S$ is that described for the algebra $R''$ in Corollary 19 but with occurrences of $p$ replaced by $x$.
\prskip

However, there is one important philosophical difference in our statement and that of Corollary 19. We have opted for a more informal statement. George Bergman's description of the basis $\mathcal{A}$ is described (very precisely, to avoid any possible ambiguities in terms such as ``words'' or ``expressions'') in terms of a certain subset $\mathcal{B}$ of the free algebra $T$ on $B \cup \{q\}$, which is then mapped faithfully to our $\mathcal{A}$ under the natural algebra homomorphism $T \rightarrow S$.
\end{proof}
\prskip

Our algebra $S$ is generated by $B \cup \{q\}$, and hence members of $S$ are linear combinations of words in this generating set, where by ``\emph{word}'' we simply mean a product of members of the generating set. Given  such an ``\emph{expression}'', to write it as a linear combination of the basis elements in $\mathcal{A}$, we apply the following \textbf{reduction rules}: repeatedly replace subwords $x^3$ by $0$, subwords $xqx$ by $x$, and subwords $qxq$ by $q$.  We need not worry about replacing a subword $uv$ for $u, v \in B$ according to the strict formalism in Corollary 19 of \cite{B1}, because in our case $uv$ is already in $B$ unless it is 0. In the former case, just leave $uv$ unchanged; in the latter, drop the word completely.
\prskip

We call the unique expression of a member of $S$ as a linear combination of basis words described in the proposition its \textbf{normal form}. This applies, in particular, to any word  in the letters $x$ and $q$. Thus the normal form of $q^2xqxq^3x^2q$ is $q^4x^2q$ \ \,(just keep replacing subwords $xqx$ by $x$, and $qxq$ by $q$). \textbf{In future, when we refer to a basis word $w$ in S, we shall implicitly assume $w$ is written in normal form.}
\prskip

From our earlier equivalent condition for unit-regularity, the following must hold:
\pfskip

\begin{proposition}\label{P:unit-regular}
If our $x \in S$ is unit-regular in $S$, then there are elements of~ $S$
 \begin{align}
\alpha  \ &= \ (1 - xq)(\sum a_iw_i)(1 - qx) \notag \\
\beta  \  &= \ (1 - qx)(\sum b_jy_j)(1 - xq), \notag
\end{align}
where the $w_i$ are distinct basis words of the form $1,q,q^2$ or $qzq$ for some (nonempty) word $z$, the $y_j$ are distinct basis words of the form $1,x,x^2$, or $xzx$, and the $a_i, b_j$ are nonzero scalars in $F$, such that
\[
 \alpha \beta = 1 - xq \ \ \ \mbox{and} \ \ \ \beta \alpha = 1 - qx.
\]
\end{proposition}
\pfskip

\begin{proof} We know that unit-regularity of $x$ requires the existence of members of $S$ of the form $\alpha = (1 - xq)u(1 - qx)$ and $\beta = (1 - qx)v(1 - xq)$ \ \,(for some $u,v \in S$) such that $\alpha \beta = 1 - xq$ and $\beta \alpha = 1 - qx$. In normal form, write $ u = \sum a_iw_i$ and $v = \sum b_jy_j$ as a linear combination of basis words. Inasmuch as $(1 - xq)x = 0 = x(1 - qx)$, a word $w_i$ that begins or ends in $x$ will be annihilated in the expansion of $\alpha$. Likewise, since $(1 - qx)q = 0 = q(1 - xq)$, any word $y_j$ that begins or ends in $q$ will be annihilated in the expansion of $\beta$. Thus we can assume that the $w_i$ and $y_j$ have the stated form.
\end{proof}
\prskip

Our strategy is to deny unit-regularity of $x$ by showing that even the equation $\alpha \beta = 1 - xq$ in Proposition \ref{P:unit-regular} is not possible (so $x \in S$ does not even have an inner inverse that is one-sided invertible). To do this, we need to examine in detail products of basis words and how certain words in the expansion of $\alpha \beta$ must occur at least twice.
\prskip

Recall that the only reductions required to put a word in letters $x, q$  in normal form are (repeated) uses of replacing a subword $x^3$ by $0$, a subword $xqx$ by $x$, and a subword $qxq$ by $q$. The product $yz$  of two basis words (in normal form) is either $0$ or is again a basis member (in normal form) after possibly one further reduction at the interface of $y$ and $z$. The product $yz$ will \emph{\textbf{involve reduction}} when $y = y'st$ and $z = sz'$, or $y = y's$ and $z = tsz'$, where $s,t$ are distinct members of $\{x,q\}$.  In either case, the reduction simply involves deleting the last letter of $y$ and the first letter of $z$. For instance, $(q^3x^2q)(xq^4x^2) = q^3x^2q^4x^2$ in normal form.  Note that once one reduction is made, no further reductions occur. For example, suppose $y = y'x$ and $z = qxz'$ with $yz \not= 0$, so that after one reduction we have $z = y'xz'$. Since $z$ is in normal, either $z' = 1$ or  $z' = xz''$ where $z''$ is in normal form. In the former case $y'xz' = y$ is in normal form, while in the latter $y'xz' = y'x^2z''$ is also in normal form.
\prskip

\noindent \textbf{Notation. For the remainder of the paper, we fix some $a_i,w_i$ and $b_j, y_j$ as in the statement of Proposition \ref{P:unit-regular}.}
\prskip

Let $\mathcal{L} = \{w_i\}$ and $\mathcal{R} = \{y_j\}$, where for convenience we won't formally introduce sets $I, J$  for the homes of the indices $i,j$. Let $\mathcal{C}$ denote the set of nonzero words expressed in normal form that
occur in the expansion of $\alpha \beta$ and begin in $q$ and end in $x$. To be clear,  by ``\emph{the expansion of $\alpha \beta$}'' we mean before one collects terms, but to simplify matters we  may as well take the product of the last $(1-qx)$ in $\alpha$ with the first $(1 - qx)$ in $\beta$ to be $1 - qx$ (it is idempotent). Then if there are $m$ terms $a_iw_i$ and $n$ terms $b_jy_j$, the formal expansion of $\alpha \beta$ involves $8mn$ terms.
\prskip

Observe that $1$ must occur as some $w_i$, say $w_1$, and as some $y_j$, say $y_1$, otherwise $1$ can't appear in $\alpha \beta = 1 -  xq$. Therefore  $ 1 - xq - qx + xq^2x$ is part of the expansion of $\alpha \beta$ because this comes from multiplying out $(1 - xq)(a_1w_1)(1)(a_1^{-1}y_1)(1 - qx)$. Hence $qx \in \mathcal{C}$. In particular, the set $\mathcal{C}$ is nonempty. Now each pair $(w,y) \in \mathcal{L} \times \mathcal{R}$ produces at most two words in the expanded $\alpha \beta$ that, after reduction, belong to $\mathcal{C}$:
\begin{align}
    wy \ \ \ & \ \ \ \ \ \mbox{which we call a \textbf{type I} word} \notag \\
    wqxy   & \ \ \ \ \ \mbox{which we call a \textbf{type II} word}. \notag
\end{align}
Some of these words may be zero, but otherwise the only exception to these two types not producing an element of $\mathcal{C}$ (again after reduction) is for the type I word $wy$ when $w = 1$ or $y = 1$.
\prskip

\begin{lemma} \label{L:types}
Here is what type I and II words look like in normal form.
\begin{enumerate}
\item A type I  word $wy$ is zero exactly when $w$ ends in $x^2q$ and $y$ begins in $x^2$.
\item A type II word $wqxy$ is zero exactly when $y$ begins in $x^2$.
\item A nonzero type I word $wy$ involves reduction exactly when $w$ ends in $xq$ and $y$ begins with $x$, or $w$ ends in $q$ and $y$ begins in $xq$. The reduced word is obtained by deleting $q$ and $x$.
\item A nonzero type II word never involves reduction.
\end{enumerate}
\end{lemma}
\pfskip

\begin{proof}
\noindent (1) The only way $wy = 0$ is when reduction takes place at the interface of $w$ and $y$, and after deleting the last letter of $w$ (it must be $q$) and the first letter of $y$ (it must be $x$),  at the new interface we are left with $x^m$ for some $m \ge 3$. Therefore, after the deletions, we must be left with $x^2$ at the end of $w$ and a single $x$ at the beginning of $y$.
\prskip

(2), (3), and (4) follow similarly.
\end{proof}
\sectskip

\section{Main result}
\introskip

Here is our main result.
\prskip

\begin{theorem} \label{T:main}
Let $F[x]$ be the polynomial ring in the indeterminate $x$ and over a field $F$, and let $R = F[x]/(x^3)$. Let $S = R\langle q \, | \, xqx = x,\, qxq = q \rangle$  be the algebra obtained from $R$ by freely adjoining a generalised inverse $q$ of $x$. Then $x$ is a nilpotent regular element of $S$ which is not unit-regular in $S$.
\end{theorem}
\pfskip

We now proceed to the key elements of the proof via two lemmas. Order the set of words in $\mathcal{C}$ by the left lexicographic order, taking $q > x$. 
Then $\mathcal C$ is a finite set with a total order, so there is a largest word $\tau$ in $\mathcal C$. We need to analyze the ways in which the word $\tau$ can appear as  type I and II products coming from $\mathcal L\times \mathcal R$. These arguments usually take the form of working out what the product looks like in reduced form (and what reduction was involved) and then using the observation (from uniqueness of the normal form) that if a word $z$ in normal form is written as a product $uv$ of two words in normal form in which the product does not involve reduction, then $u$ and $v$ must be a two-part partitioning of the string $z$. Also, our arguments often play off $\tau$ occurring as a type I word (respectively, type II word) against the corresponding type II word (respectively, type I word) being bigger unless certain conditions are met.
\pfskip

\begin{lemma} \label{lem:waystauappear}
Let
\[
\tau = q^{i_1}x^2q^{i_2}x^2 \cdots q^{i_n}x^c, \ \ \ \ (n \ge 1, \, i_1 \ge 1, i_2,\cdots,i_n \ge 2, \, c \in \{1,2\})
\]
be the largest element of $\mathcal{C}$ with respect to the lexicographic order. For $\tau$ to  occur as a type I or II word (after reduction) from $(w,y) \in \mathcal{L} \times \mathcal{R}$, only the following are possible:
\begin{enumerate}
 \item $\tau$ is the type I word $wy$ with no reduction and
 \[
 w = q^{i_1}x^2q^{i_2}x^2 \cdots q^{i_r}, \ \ y = x^2q^{i_{r+1}}x^2 \cdots q^{i_n}x^c , \ \ 1 \le r \le n.
 \]
 \item  $\tau$ is the type I word $wy$ with reduction and
 \[
     w = q^{i_1}x^2 \cdots q^{i_{(r-1)}}x^2q^a, \ \ y = xq^bx^2q^{i_{r+1}}x^2 \cdots q^{i_n}x^c, \qquad  1\le r\le n,
 \]
 where $a + b - 1 = i_r, \, a \ge 1, b \ge 2$ and either $b > 2$, or $b = 2$ and $i_t > 2$ for some $t > r$.
 \item $\tau$ is the (nonzero) type II word
 $wqxy$ where either
 \[
   w = q^{i_1}x^2q^{i_2}x^2 \cdots q^{i_r-1}, \ \ y = xq^{i_{r+1}}x^2 \cdots q^{i_n}x^c, \qquad 1\le r < n ,
 \]
 and $i_t = 2$ for all $t > r$, or
 \[
   w = q^{i_1}x^2q^{i_2}x^2 \cdots q^{i_n-1}, \ \ y = x  .
 \]
  \end{enumerate}
 \end{lemma}
\pfskip

\begin{proof}  We begin by eliminating the possibility that $\tau$ could come from a pair $(w,y)$ with $w = 1$ or $y = 1$. Observe that we must have $q \in \mathcal{L}$ and $x \in \mathcal{R}$, otherwise the term $qx$, which comes from the type II product involving the pair $(1,1)$, could not be cancelled. However, $qx$ is not cancelled by any type II word coming from a pair $(w,y) \neq (1,1)$, so $qx$ must be cancelled by a type I word $wy$. If there is no reduction involved, then $w = q$ and $y = x$. However, if $w = w'q$ and $y = xy'$ involves reduction, then $qx = wy = w'y'$ implies $w' = q, \,y' = x$, whence $w = q^2$ and $y = x^2$, contradicting any reduction. Thus $q \in \mathcal{L}$ and $x \in \mathcal{R}$. Now suppose $\tau$ comes from a pair $(w,1)$, which must be the type II product $wqx$ because $\tau$ begins in $q$ and ends in $x$. However, the type II product from $(w,x)$ is $wqx^2 > \tau$, contradiction. Similarly, if $\tau$ comes from $(1,y)$, it must be the type II product $\tau = qxy$, but the pair $(q,y)$ produces the bigger type II product $q^2xy$.   Henceforth, we can assume $\tau$ comes only from pairs $(w,y)$ with $w \neq 1, y \neq 1$.
\prskip

Suppose that $(w, y)\in \mathcal L \times \mathcal R $ gives rise to $\tau$ (through a type I or II word). If $y$ begins in $x^2$, then the type II word $wqxy$ is $0$, so $\tau = wy$ is the type I word without reduction by Lemma \ref{L:types}\,(3). Therefore, we must have the form stated in (1) because $w$ ends in $q$ and $y$ begins in $x$.
\prskip

Next consider the case where $y = xy'$, where $y'$ does not start in $x$. If $y'=1$, then $wy = wx < wqxy= wqx^2$. Therefore $\tau $ can only be the type II word $wqxy$, because $\tau$ is the largest element of $\mathcal C$ in the lexicographic order. Thus
\[
(w, y) = (q^{i_1}x^2q^{i_2}x^2 \cdots q^{i_n-1},\, x)
\]
and we are in the second instance of case (3).
\prskip

It therefore suffices to consider the case where
\[
(w, y )\, = \,(w'q^a,\, xq^by'' )
 \]
with $a\ge 1$, \,$b\ge 2$, \,$w'$ not ending with $q$, and $y''$  not starting with $q$. Firstly, suppose $\tau$ occurs as the type II word  $ wqxy = w'q^{a+1}x^2q^b y''$. Then there is some $1\le r < n$ such that $a+1=i_r$,\, $w'q^{a+1}= q^{i_1}x^2 \cdots x^2q^{i_r}$, and $x^2q^by'' = x^2q^{i_{r+1}}x^2 \cdots q^{i_n}x^{c} $. We thus obtain
\[
(w, y) = (q^{i_1}x^2 \cdots q^{i_r-1}, \, xq^{i_{r+1}}x^2 \cdots q^{i_n}x^{c} ).
\]
But now observe that the type I word
\[
wy \ = \ q^{i_1}x^2 \cdots q^{i_{r-1}}x^2q^{(i_r + i_{r+1}-2)}x^2 q^{i_{r+2}} x^2 \cdots q^{i_n}x^{c},
\]
which we have written in normal form according to Lemma \ref{L:types}\,(3), will be greater than $\tau$ unless $i_t=2$ for all $t > r$. Indeed, $i_{r+1} = 2$ otherwise the exponent of the $r$\,th group of $q$'s will be greater in $wy$ than in $\tau$, and if $i_t > 2$ for some $t \ge r+2$ , the least such $t$ will give a bigger exponent of $q$'s than the matching group of $q$'s in $\tau$ (because the groupings after the $r$\,th have been pulled back one place in $wy$). Therefore to obtain $\tau$ as a type II word, we must be in the first instance of case (3).
\prskip

Secondly, if $\tau = w'q^{a+b-1}y''$ is the type I word obtained from $(w,y)$ after reduction, then there exist $1\le r \le n $ such that $a+b-1 = i_r$, \, $w'q^{a+b-1} = q^{i_1}x^2 \cdots x^2q^{i_r}$, and $y'' = x^2q^{i_{r+1}}x^2 \cdots q^{i_n}x^{c} $. Therefore
\[
(w, y)\ = \ (q^{i_1}x^2 \cdots q^{i_{r-1}}x^2q^a,\, xq^bx^2q^{i_{r+1}}x^2 \cdots q^{i_n}x^{c} ).
\]
From this $(w,y)$ we also get the type II word
\[
 wqxy \ = \ q^{i_1}x^2 \cdots q^{i_{r-1}}x^2q^{a+1}x^2q^b x^2q^{i_{r+1}}x^2 \cdots q^{i_n}x^{c}.
\]
But now observe that the latter word is greater than $\tau$ unless $a + 1 < i_r$ (whence  $b > 2$), or $a + 1 = i_r$ (whence $b = 2$) and there is some $t > r$ such that $i_t > 2$. Hence we are in case (2).  This concludes the proof of the lemma.
\end{proof}
\prskip

\begin{lemma}  \label{lem:lex-order}
The greatest element $\tau$ of $\mathcal{C}$ can occur at most once in the form {\rm(}2{\rm)} or {\rm(}3{\rm)}
 of Lemma \ref{lem:waystauappear} but not both.
\end{lemma}
\pfskip

\begin{proof}  We first  show that there is at most one pair of the form given in Lemma \ref{lem:waystauappear}(2). Suppose  we have two different pairs $(w, y)$, $(w_1, y_1)$ of that form. Note that if $y = y_1$, then by the nature of the reduction that is taking place
in the two products $wy = w_1y_1 \,\, (=\tau )$, we must have $w = w_1$. Hence either $y > y_1$ or $y_1 > y$. If $y >y_1$, then the pair $(w_1, y)$ gives the type I word $w_1y$ in $\mathcal C$ which (after reduction) is bigger than $\tau$. On the other hand, if  $y_1 > y $, then the pair $(w, y_1 )$ gives the type I word $wy_1$ in $\mathcal C$  which (after reduction) is bigger than $\tau$. In either case we get a contradiction. This establishes  that there is at most one pair of form (2) in Lemma \ref{lem:waystauappear}.
\prskip

Next we show that there is at most one pair of the form given in Lemma \ref{lem:waystauappear}(3).  Suppose  we have two different  pairs $(w'q^{i_r-1}, xy')$ and $(w'_1q^{i_s-1}, xy '_1)$ of that form. Without loss of generality, we can  suppose that $s > r$. Then from the condition that $i_t=2$ for all $t > r$, we must have  $y'_1<y'$.  So we arrive at a contradiction after considering the type II word associated with pair $(w'_1q^{i_s-1}, xy' )$, which gives the element in $\mathcal C$
\[
 q^{i_1}x^2 \cdots x^2q^{i_s}x^2q^{i_{r+1}}x^2 \cdots q^{i_n}x^{c} >\tau.
\]
\prskip

Finally assume that we have a pair $(w, y)$ of the form given in Lemma \ref{lem:waystauappear}(2), with corresponding $(a,b)$  satisfying $a+b-1 = i_r$, and a pair $(w_1, y_1) = (w'q^{i_s-1}, xy ')$ of the form given in Lemma \ref{lem:waystauappear}(3). Assume first that $r = s$. The only  way this is possible is to have $b > 2$ and $i_t= 2$ for all $t >r $. In this case, both the type I and type II words arising from the pair $(w_1, y) = (w'q^{i_r-1}, y )$ are bigger than $\tau$ in $\mathcal C$. Therefore $r \neq s$.
\prskip

Suppose now that $r < s$. Then $y = xq^b x^2 q^{i_{r+1}}x^2 \cdots q^{i_n}x^{c}$, with $b \ge 2$,  and either $s < n$ and $y_1 = xq^{i_{s+1}}x^2 \cdots q^{i_n} x^{c}$  and  $i_t = 2$ for all $t > s$,
or $s = n$ and $y_1 = x$. But now the pair $(w_1, y)=(w'q^{i_s-1}, y )$ gives rise to the type II word
\[
 q^{i_1}x^2 \cdots x^2 q^{i_s} x^2 q^b x^2 q^{i_{r+1}}x^2 \cdots q^{i_n} x^{c} > \tau,
\]
a contradiction. Hence we must have $r > s$. Necessarily from the form of (3), we have $i_r = 2$ (and so $b = 2$) and $i_t = 2$ for all $t > r$. This clearly violates the stipulated form in (2). Thus this final case is  not possible either,
which establishes there is at most one pair $(w,y)$ that  produces either (2) or (3).
\end{proof}
\pfskip

We are ready for:
\introskip

\noindent \textbf{\emph{Proof of Theorem}} \ref{T:main}.
\pfskip

\noindent The closing argument of our proof is the most enjoyable part. Suppose $x$ is unit-regular in our ring $S$. By Proposition \ref{P:unit-regular}, there are elements $\alpha, \beta \in S$ of the form described such that
\[
     \alpha \beta \ = \ 1 - xq.
\]
After expanding $\alpha \beta$ as a linear combination of words in $x$ and $q$ (but not necessarily in normal form and allowing for repetition of words), Lemmas \ref{lem:waystauappear} and \ref{lem:lex-order} tell us how the largest (in the lexicographic order) member $\tau$ in the subset $\mathcal{C}$ (of nonzero words from the expansion, expressed in normal form, and beginning in $q$ and ending in $x$) can occur. Inasmuch as $\tau$ definitely resides in $\mathcal{C}$, it must occur in the expansion of $\alpha \beta$ \textbf{at least twice}. Otherwise the linear combinations of $\tau$ could not be zero in the final simplification of $\alpha \beta$ to $1 - xq$ in normal form, which involves \textbf{no} terms from $\mathcal{C}$. Therefore, from Lemma \ref{lem:lex-order}, it must be that $\tau$ occurs at least once as a type I word $\tau = wy$ without reduction.  But now when we form the type II word from the pair $(w,1)$ we have
\[
        wqx \ > \ wy \ = \ \tau
\]
because $y$ begins in $x$ and $q > x$. This contradiction shows $x$ cannot be unit-regular in $S$. Our mission is accomplished. \qed

\bigskip

We close this section by noting that our result implies the non-separativity of $S$.
\introskip

\begin{corollary}
 The ring $S$ is non-separative.
\end{corollary}

\begin{proof}
 Observe that, since $x^3= 0$, we have
 $$(1-xq) +x(1-xq)q+ x^2(1-xq)q^2 = 1 = (1-qx) +q(1-qx)x+ q^2(1-qx)x^2 ,$$
 so that both $1-xq$ and $1-qx$ are full idempotents in $S$ (they generate $S$ as a two-sided ideal). Therefore we have
 $$S = (1-xq) S \oplus xqS = (1-qx)S \oplus qxS $$
 with $xqS\cong qxS$, and $xqS$ is isomorphic to both a direct summand of copies of $(1-xq)S$ and a direct summand of copies of $(1-qx)S$.
 Since $(1-xq)S\ncong (1-qx)S$ by our main result, it follows from \cite[Lemma 2.1]{A2} that $S$ is non-separative.
 \end{proof}
\sectskip

\section{Another approach}
\introskip

Here we unify the Nielsen--\v{S}ter example described in the introduction with our own example. It is always gratifying when two camps have worked quite independently of each other, with different approaches, and yet come up with the same counter-example.
\prskip

\begin{proposition}\label{P:ringS}
Let $R= F\langle a,b\rangle $ be the free $F$-algebra on $a,b$. Let $S = F\langle q,x \mid q=qxq, x=xqx \rangle $.
Let $I = R(1-ba)$ and let $T$ be the subalgebra of $M_2(R)$ given as
\[
 T = \left[\begin{array}{cc}
   R  & I \\
   R & F+I \end{array}\right],
\]
where $F + I$ means $F1 + I$. Then there is a natural isomorphism $S \cong T$ under which
 $q \mapsto Q = {\small \left[\begin{array}{cc}
   b  & 1-ba \\
   0 & 0 \end{array}\right]}$
and $x \mapsto X = {\small \left[\begin{array}{cc}
   a  & 0 \\
   1 & 0 \end{array}\right]}$.
Moreover, the same conclusion holds if for some fixed $n \ge 3$, we impose the extra relation  $x^n = 0$ on $S$ and replace $R$ by $F\langle a, b \mid a^{n-1} = 0 \rangle$.
\end{proposition}
\pfskip

\begin{proof}
Let $\varphi \colon S \to T$ be the homomorphism defined by $\varphi (q) =  Q$ and $\varphi (x) = X$, which is well-defined because $QXQ = Q$ and $XQX= X$. By the argument in Proposition \ref{P:normal form}, Bergman's normal form for $S$ provides a basis consisting of words alternating in powers of $q$ and powers of $x$, but with the restriction that powers to exponent 1 can occur only at the beginning or end. From this we see that  $qxSqx$ is freely generated by the elements $q^2x$ and $qx^2$ because $q^ix^j = (q^2x)^{i-1}(qx^2)^{j-1}$ and $qxSqx$ has a basis consisting of $qx$ (its identity) and all words in normal form that begin in $q$ and end in $x$. Since
\[
  \varphi(qx^2) = \left[\begin{array}{cc}
   a  & 0 \\
   0 & 0 \end{array}\right], \ \
 \varphi(q^2x) =   \left[\begin{array}{cc}
   b  & 0 \\
   0 & 0 \end{array}\right]
\]
it follows that $\varphi$  induces an isomorphism from $qxSqx$ onto
$ {\small \left[\begin{array}{cc}
   R  & 0 \\
   0 & 0 \end{array}\right]}$.
\prskip

It is easily checked that an $F$-basis for $(1-qx)S(1-qx)$ is given by
 $$\{ 1-qx \} \cup \{(1-qx) x^{j_0}(q^2x)^{i_1}(qx^2)^{j_1}\cdots (qx^2)^{j_{n-1}} q^{i_n} (1-qx) \} ,$$
 where $j_0,i_1,j_1,\dots , i_n \ge 1 $. The image by $\varphi$ of this basis is
 $$\{ e_{22} \} \cup \{ e_{22} a^{j_0-1} b^{i_1}a^{j_1} \cdots a^{j_{n-1}} b^{i_n -1} (1-ba) \} ,$$
 which is an $F$-basis of $e_{22} T e_{22}$. (Here,  $e_{ij}$ denote the usual matrix units in $T$.)
 Similar arguments show that a basis of $(1-qx) S qx $ is mapped onto a basis of $e_{22}Te_{11}$ and a basis of
 $qxS(1-qx)$ is mapped onto a basis of $e_{11}Te_{22}$. Thus $\varphi$ is an isomorphism.
 \prskip

 Alternatively, having verified (as in the first paragraph) that $\varphi$ induces an isomorphism of $eSe$ onto $fTf$, where $e, f$ are the idempotents $qx, e_{11}$ in $S$ and $T$ respectively, we could complete the proof as follows. Firstly, we observe that $S$ is a prime ring. Note that for any nonzero $z \in S$, either $qz$ or $xz$ is nonzero.  For if $qz = 0$ there must be reduction involved with products of $q$ and all words in $z$ of greatest length,  so all such words must begin in $x$.  And if $xz = 0$ also, they must all begin in $q$, contradiction. Hence for $z \neq 0$, either $qz \neq 0$ or $qxz \neq 0$. Similarly for $0 \neq y \in S$, either $yq \neq 0$ or $yxq \neq 0$. Hence in showing $ySz \neq 0$, we can assume $y$ is a left multiple of $q$ and $z$ is a right multiple of $q$. But now $yz \neq 0$ because there is no reduction involved in multiplying a basis word in $y$ of greatest length with one in $z$ of greatest length. Thus $S$ is prime.   If $K = \ker \varphi \neq 0$, primeness of $S$ gives $Ke \neq 0$, whence $Ke \subseteq (1-e)Se$ because $\varphi$ is faithful on $eSe$. This makes the left ideal $Ke$ nilpotent, contradiction. Hence $K = 0$. Also $\varphi$ is onto because in addition to $fTf \subseteq \varphi(S)$, we have that
 \begin{align}
   &fT(1-f) = fTfQ(1-f), \ (1-f)Tf = (1-f)XfTf, \notag \\
   &(1-f)T(1-f) = (1-f)XfT(1-f) + (1 - f)F \notag
 \end{align}
 are all in the image of $\varphi$.
 \prskip

 If we impose $x^n = 0$, the basis words in the  normal form for $S$ have the extra restriction that the only  powers of $x$ allowed are $x^i$ for $i = 1,2,\ldots, n-1$. Just the obvious extension of Proposition \ref{P:normal form}. And for words in $a,b$, when we impose $a^{n-1} = 0$, powers of $a$ allowed are $a, a^2,\ldots, a^{n-2}$. It is easy to show that when $n \ge 3$, the same $\varphi$ sends the basis for $S$ to the basis for $T$, so our first proof also works here. Alternatively, we can check primeness of $S$ when $n \ge 3$. For any nonzero $z$ with $qz = xz = 0$, we deduce from $qz = 0$ that all the greatest length terms in $z$ must begin in $x$. Then $xz = 0$ shows these terms begin in $x^{n-1}$. But now since $n \ge 3$  there is no reduction in left multiplying such words by $q$, contradicting $qz = 0$. Hence $qz = xz = 0$ implies $z = 0$. Similarly, $yq = yx = 0$ implies $y = 0$. Thus $S$ is prime and the statements in the proposition remain true when $x^n = 0$ and $n \ge 3$.
\end{proof}
\prskip

\begin{remark}\label{R:n=2}
As it stands, Proposition \ref{P:ringS} fails for $n = 2$. The homomorphism $\varphi$ is still onto but has a nonzero kernel because $\varphi(1 - qx - xq + xq^2x) = I - QX - XQ + XQ^2X = 0$. 
Our proof fails because $S$ is no longer prime {\rm(}$1 - qx - xq + xq^2x$ is a central idempotent{\rm)}. Note when $n = 2$, 
we have $a = 0$ and $R$ is the polynomial algebra $F[b]$. The problem with $a = 0$ is that $1$ and $1 - ba$ become the same. 
However, if we set $T' = M_2(F[b]) \times F$, we can show $S \cong T'$ via the {\rm(}unital{\rm)} mapping that sends
\[
   q \mapsto Q = \left({\Small \left[\begin{array}{cc}
   b  & 1 \\
   0 & 0 \end{array}\right]}, \,0 \right)\ \ \ \mbox{and} \ \ \ 
    x \mapsto X = \left({\Small \left[\begin{array}{cc}
   0 & 0 \\
   1 & 0 \end{array}\right]}, \,0\right).
\]   
\end{remark}
\prskip

 From now on, we set $R = F\langle a,b \mid a^2 = 0 \rangle $. A basis for $R$ is the set $\mathcal B= \{ a^{i_0}b^{i_1}ab^{i_2}a\cdots b^{i_{r-1}}a^{i_r} \}$,
 where $i_0,i_r \in  \{ 0,1 \}$, $r \ge 0$ and $i_1,\dots ,i_{r-1}\ge 1$.
A product of two basis elements $\alpha \beta$ is a basis element (without reduction) or $0$, and it is zero
 if and only if $\alpha $ ends in $a$ and $\beta $ starts in $a$. Observe that $b$ is a non-zero-divisor in $R$.
\prskip

\begin{theorem}\label{thm:main}
Let $T$ be the algebra in \ref{P:ringS} for the choice of $n = 3$. Then $X$ is a regular nilpotent element which is not unit-regular in $T$.
\end{theorem}
\pfskip

\begin{proof}
Suppose $X$ is unit-regular in $T$. Then $(1 - XQ)T \cong (1 - QX)T$ and therefore there exist $u \in I, t,v \in F + I, z\in R$ such that
\[
\left[\begin{array}{cc}
    0  & 0   \\
    z  & t
   \end{array}\right]
\left[\begin{array}{cc}
   0 & u   \\
   0 & v
   \end{array}\right]  =
\left[\begin{array}{cc}
    0  & 0   \\
    0  & 1
   \end{array}\right], \ \
\left[\begin{array}{cc}
   0   &  u  \\
   0   &  v
   \end{array}\right]
\left[\begin{array}{cc}
   0   & 0   \\
   z   & t
   \end{array}\right] =
\left[\begin{array}{cc}
 1 - ab & -a(1 - ba)   \\
   -b   &  ba
   \end{array}\right].
\]
In particular, we have $zu+tv= 1$, $vz= -b$, $vt= ba$. Hence
\[
v= (vz)u+ (vt)v= bv_1
\]
for some $v_1\in R$ and, since $v\in F+I$, we conclude that $v\in I=R(1-ba)$. Therefore $v_1= v_2(1-ba)$ for some $v_2\in R$.  Inasmuch as $b$ is a non-zero-divisor in $R$, from $-b= vz= bv_2(1-ba)z$ we deduce that the equation $c(1-ba)d=1$ has a solution $c,d\in R$. Consider the homomorphism $\pi : R \longrightarrow M_2(F)$ obtained by mapping
$a \mapsto {\Small
\left[\begin{array}{cc}
       0 & 0 \\
       1 & 0
       \end{array}\right]}$
and
$b \mapsto {\Small
\left[\begin{array}{cc}
       0 & 1 \\
       0 & 0
       \end{array}\right]}$.
From $\pi(c)\pi(1 - ba)\pi(d) = \pi(1)$ this implies the equation
$C {\Small
\left[\begin{array}{cc}
       0 & 0 \\
       0 & 1
       \end{array}\right]}D = I$
has a solution $C,D \in M_2(F)$. We have reached a desired contradiction (look at the determinant of each side). This completes the proof.
\end{proof}
\prskip

\begin{corollary}
Let $S = F\langle q, x \, | \, x^3 = 0, xqx = x, qxq = q \rangle$. Then $x$ is a regular nilpotent element of $S$ which is not unit-regular.
\end{corollary}
\pfskip

\begin{proof} Apply the isomorphism in Proposition \ref{P:ringS}.
\end{proof}

\bigskip

\section{Acknowledgments}
\prskip

We are indebted to George Bergman for his many helpful comments and suggestions, and his generosity in revising an earlier version of \cite{B1} to make the proof of Proposition \ref{P:normal form} more accessible to our reader.   We also thank T. Y. Lam for helpful discussions and friendly correspondence.

\sectskip

\end{document}